\documentclass[12pt]{amsart}
\usepackage[latin1]{inputenc}
\usepackage{color}
\usepackage{enumerate}
\usepackage{amssymb}
\usepackage[all]{xy}
\usepackage{hyperref}

\newtheorem{thm}{Theorem}[section]
\newtheorem{cor}[thm]{Corollary}
\newtheorem{lem}[thm]{Lemma}
\newtheorem{prop}[thm]{Proposition}
\theoremstyle{definition}
\newtheorem{defn}[thm]{Definition}
\theoremstyle{remark}

\newtheorem{prob}{Problem}

\numberwithin{equation}{section}
\newcommand{\norm}[1]{\left\Vert#1\right\Vert}
\newcommand{\abs}[1]{\left\vert#1\right\vert}
\newcommand{\set}[1]{\left\{#1\right\}}

\newcommand{\To}{\longrightarrow}

\newcommand{\nat}{\mathbb{N}}

\usepackage{color}
\usepackage{enumerate}
\begin{document}
\setcounter{tocdepth}{1}


\title{On projective Banach lattices of the form $C(K)$ and $FBL[E]$}
\author[A.\ Avil\'es]{Antonio Avil\'es}
\address{Universidad de Murcia, Departamento de Matem\'{a}ticas, Campus de Espinardo 30100 Murcia, Spain.}
\email{avileslo@um.es}

\author[G. Mart\'inez-Cervantes]{Gonzalo Mart\'inez-Cervantes}
\address{Universidad de Murcia, Departamento de Matem\'{a}ticas, Campus de Espinardo 30100 Murcia, Spain.}
\email{gonzalo.martinez2@um.es}

\author[J.D. Rodr\'iguez Abell\'an]{Jos\'e David Rodr\'iguez Abell\'an}
\address{Universidad de Murcia, Departamento de Matem\'{a}ticas, Campus de Espinardo 30100 Murcia, Spain.}
\email{josedavid.rodriguez@um.es}

\thanks{Authors supported by project MTM2017-86182-P (Government of Spain, AEI/FEDER, EU) and project 20797/PI/18 by Fundaci\'{o}n S\'{e}neca, ACyT Regi\'{o}n de Murcia. Third author supported by FPI contract of Fundaci\'on S\'eneca, ACyT Regi\'{o}n de Murcia.}



\begin{abstract}
We show that if a Banach lattice is projective, then every bounded sequence that can be mapped by a homomorphism onto the basis of $c_0$ must contain an $\ell_1$-subsequence. As a consequence, if Banach lattices $\ell_p$ or $FBL[E]$ are projective, then $p=1$ or $E$ has the Schur property, respectively. On the other hand, we show that $C(K)$ is projective whenever $K$ is an absolute neighbourhood retract, answering a question by de Pagter and Wickstead.\end{abstract}

\maketitle

\setlength{\parskip}{4mm}

\section{Introduction}

In this paper we continue the program proposed by B. de Pagter and A. W. Wickstead \cite{dPW15} of studying the projective Banach lattices.

\begin{defn}\label{lambdaprojdef} Let $\lambda>1$ be a real number. A Banach lattice $P$ is \textit{$\lambda$-projective} if whenever $X$ is a Banach lattice, $J$ a closed ideal in $X$ and $Q \colon X \longrightarrow X/J$ the quotient map, then for every Banach lattice homomorphism $T \colon P \longrightarrow X/J$, there is a Banach lattice homomorphism $\hat{T} \colon P \longrightarrow X$ such that $T = Q \circ \hat{T}$ and $\|\hat{T}\| \leq \lambda\norm{T}$.
\end{defn}

A Banach lattice is called \textit{projective} in \cite{dPW15} if it is $(1 + \varepsilon)$-projective for every $\varepsilon > 0$. For a more intuitive terminology, and by analogy to similar notions in Banach spaces, we will call this \textit{$1^+$-projective} instead of just projective. Note that if $P$ is $\lambda$-projective, then $P$ is $\mu$-projective for every $\mu \geq \lambda$. We will call a Banach lattice \textit{$\infty$-projective} if it is $\lambda$-projective for some $\lambda> 1$. It is clear that, in the case of $\infty$-projective, $Q$ can be taken any surjective Banach lattice homomorphism.

The notion of free Banach lattice was also introduced in \cite{dPW15}.  If $A$ is a set with no extra structure, the free Banach lattice generated by $A$, denoted by $FBL(A)$, is a Banach lattice together with a bounded map $u \colon A \longrightarrow FBL(A)$ having the following universal property: for every Banach lattice $Y$ and every bounded map $v \colon A \longrightarrow Y$ there is a unique Banach lattice homomorphism $S \colon FBL(A) \longrightarrow Y$ such that $S \circ u = v$ and $\norm{S} = \sup \set{\norm{v(a)} : a \in A}$. The same idea is applied by A. Avil\'{e}s, J. Rodr\'{i}guez and P. Tradacete to define the concept of the free Banach lattice generated by a Banach space $E$, $FBL[E]$. This is a Banach lattice together with a bounded operator $u \colon E\To FBL[E]$  such that for every Banach lattice $Y$ and every bounded operator $T \colon E \longrightarrow Y$ there is a unique Banach lattice homomorphism $S \colon FBL[E] \longrightarrow Y$ such that $S \circ u = T$ and $\norm{S} = \norm{T}$.

In \cite{ART18} and \cite{dPW15} it is shown that both objects exist and are unique up to Banach lattices isometries. A more explicit description of these spaces is given in \cite{ART18} as follows:

Let $A$ be a non-empty set. For $x \in A$, let $\delta_x \colon [-1,1]^A \longrightarrow \mathbb{R}$ be the evaluation function given by $\delta_x(x^*) = x^*(x)$ for every $x^* \in [-1,1]^A$, and for every $f \colon [-1,1]^A \longrightarrow \mathbb{R}$ define $$\|f\| = \sup \set{\sum_{i = 1}^n \abs{ f(x_{i}^{\ast})} :  n \in \mathbb{N}, \, x_1^{\ast}, \ldots, x_n^{\ast} \in [-1,1]^A, \text{ }\sup_{x \in A} \sum_{i=1}^n \abs{x_i^{\ast}(x)} \leq 1 }.$$ 

The Banach lattice $FBL(A)$ is the Banach lattice generated by the evaluation functions $\delta_x$ inside the Banach lattice of all functions $f \colon [-1,1]^A\To\mathbb{R}$ with finite norm. The natural identification of $A$ inside $FBL(A)$ is given by the map $u \colon A \longrightarrow FBL(A)$ where $u(x) = \delta_x$. Since every function in $FBL(A)$ is a uniform limit of such functions, they are all continuous (with respect to the product topology) and positively homogeneous, i.e. they commute with multiplication by positive scalars.

Now, let $E$ be a Banach space. For a function $f \colon E^\ast \To \mathbb{R}$ consider the norm $$ \norm{f}_{FBL[E]} = \sup \set{\sum_{i = 1}^n \abs{f(x_{i}^{*})} : n \in \mathbb{N}, \, x_1^{*}, \ldots, x_n^{*} \in E^{*},\text{ }\sup_{x \in B_E} \sum_{i=1}^n \abs{x_i^{*}(x)} \leq 1 }.$$

The Banach lattice $FBL[E]$ is the closure of the vector lattice in $\mathbb R^{E^*}$ generated by the evaluations $\delta_x \colon x^\ast \mapsto x^\ast(x)$ with $x\in E$. These evaluations form the natural copy of $E$ inside $FBL[E]$. All the functions in $FBL[E]$ are positively homogeneous and $weak^\ast$-continuous  when restricted to the closed unit ball $B_{E^\ast}$. An alternative approach to the construction of $FBL[E]$ has been given in \cite{Troitsky}.

In our previous work \cite{AMRA19} we answered a question by B. de Pagter and A. W. Wickstead by showing that $c_0$ is not a projective Banach lattice. In Section \ref{SchurPropertyInBanachSpacesWithProjectiveFreeBanachLattice} we exploit some of the ideas of that paper further, so that we are able to show that projective Banach lattices enjoy the following property:

\begin{thm}\label{TheoremSubsequencesEquivalentTol1}
	Let $(u_i)_{i \in \mathbb{N}}$ be a bounded sequence of vectors in an $\infty$-projective Banach lattice $X$. Suppose that there exists a Banach lattice homomorphism $T \colon X \longrightarrow c_0$ such that $T(u_i) = e_i$ for every $i \in \mathbb{N}$, where $(e_i)_{i \in \mathbb{N}}$ is the canonical basis of $c_0$. Then there is a subsequence $(u_{i_k})_{k \in \mathbb{N}}$ equivalent to the canonical basis of $\ell_1$.
\end{thm}

From this theorem we can deduce for example that no Banach lattice $\ell_p$ is $\lambda$-projective, for any $\lambda>0$ or $p\neq 1$. The prototype of $1^+$-projective Banach lattice is $FBL(A) = FBL[\ell_1(A)]$ (see \cite[Corollary 2.8]{ART18} and \cite[Proposition 10.2]{dPW15}), so it would be natural to wonder whether $FBL[E]$ might be $1^+$-projective as well for other Banach spaces $E$. We show that, for this to happen, the structure of $E$ must be very close to that of $\ell_1(A)$:

\begin{thm}\label{TheoremSchurPropertyInBanachSpacesWithProjectiveFreeBanachLattice}
Let $E$ be a Banach space. If $FBL[E]$ is $\infty$-projective, then $E$ has the Schur property (i.e. every weakly convergent sequence in $E$ converges in norm).
\end{thm}

Moreover, at the end of Section \ref{SectionSchurProperty} we provide a counterexample which shows that, in the category of nonseparable Banach spaces, the converse of this result does not hold. We still do not know if there exists a separable Banach space $E$ which has the Schur property and such that $FBL[E]$ is not $\infty$-projective.

Section \ref{ProjectivityOfC(K)} is devoted to the study of projectivity on Banach lattices of the form $C(K)$ of continuous functions on a compact space with the supremum norm. It was shown by de B. Pagter and A. W. Wickstead \cite{dPW15} that if $C(K)$ is $1^+$-projective then $K$ is an absolute neighbourhood retract, and the converse holds true if $K$ is a compact subset of $\mathbb{R}^n$. They asked whether the converse holds for general $K$ \cite[Question 12.12]{dPW15}. We solve their problem in the positive:

\begin{thm}
\label{thmANR}
If $K$ is a compact Hausdorff topological space, then $C(K)$ is $1^+$-projective if, and only if, $K$ is an absolute neighbourhood retract.
\end{thm}

\section{Schur property in Banach spaces with projective free Banach lattice}\label{SchurPropertyInBanachSpacesWithProjectiveFreeBanachLattice}
\label{SectionSchurProperty}
%
As a preparation towards Theorem~\ref{TheoremSubsequencesEquivalentTol1} we provide a criterion to obtain $\ell_1$-subsequences in the free Banach lattice $FBL(L)$. We denote the index set $L$ instead of $A$ for convenience in latter application.

\begin{lem}
\label{lemmaauximproved}
Let $L$ be an infinite set, $(x_n^*)_{ n \in \mathbb{N} }$ a sequence in $[-1,1]^L$ and $(f_n)_{ n \in \mathbb{N} }$  a sequence in $FBL(L)$ with the following properties:
\begin{enumerate}
	\item $(f_n)_{ n \in \mathbb{N} }$ converges pointwise to 0, i.e. $\lim_{n \to \infty} f_n(x^*) = 0$ for every $x^* \in [-1,1]^L$;
	\item $f_n(x_n^*)=1$ for every $n \in \nat$;
	\item For every finite set $F\subset L$ there is a natural number $n$ such that $x_n^*|_F=0$, i.e. the restriction of $x_n^*$ to $F$ is null.
\end{enumerate}
Then, for every $\varepsilon>0$ there is a subsequence $(f_{n_k})_{ k \in \nat }$ such that for every $l \in \mathbb{N}$ and for every $\lambda_1, \ldots, \lambda_l \in \mathbb{R}$, $$\norm{ \sum_{k=1}^{l}\lambda_k f_{n_k} } \geq (1 - \varepsilon)\sum_{k=1}^l |\lambda_k|.$$
\end{lem}

\begin{proof}

Fix $\varepsilon>0$ and $(\varepsilon_{ij})_{i,j=1}^{\infty}$ a family of positive real numbers such that $\varepsilon = \sum_{i,j=1}^{\infty}\varepsilon_{ij}$ and $\varepsilon_{ij} =  \varepsilon_{ji}$ for every $i,j$.

We are going to define a subsequence $(f_{m_k})_{k \in \mathbb{N}}$ of $(f_n)_{n \in \mathbb{N}}$ as follows:

Let $m_1 := 1$. Since the elements of $FBL(L)$ are continuous with respect to the product topology, there is a neighbourhood $U_{m_1}$ of $x_{m_1}^*$ such that $f_{m_1}(x^*) \in [1-\varepsilon_{11}, 1+\varepsilon_{11}]$ whenever $x^* \in U_{m_1}$.
In particular, there is a finite set $F_{m_1} \subset L$ such that $f_{m_1}(x^*)\in [1-\varepsilon_{11}, 1+\varepsilon_{11}]$ whenever $x^*|_{F_{m_1}} = x_{m_1}^*|_{F_{m_1}}$.

By property $(3)$, there exists $m_2 \in \mathbb{N}$ such that $x_{m_2}^*|_{F_{m_1}} = 0$. Since $f_{m_2}$ is continuous, there exists a finite set $F_{m_2} \supset F_{m_1}$ such that $f_{m_2}(x^*) \in [1 - \varepsilon_{22},1+\varepsilon_{22}]$ whenever $x^*|_{F_{m_2}} = x_{m_2}^*|_{F_{m_2}}$.

Suppose that we have $f_{m_1}, \ldots, f_{m_{k-1}}$ for some $k \geq 2$, and $F_{m_1}, \ldots, F_{m_{k-1}}$ finite subsets of $L$ such that $F_{m_1} \subset \cdots \subset F_{m_{k-1}}$, $x_{m_i}^*|_{F_{m_{i-1}}} = 0$ for every $i = 2, \ldots, k-1$ and $f_{m_i}(x^*) \in [1 - \varepsilon_{ii},1+\varepsilon_{ii}]$ whenever $x^*|_{F_{m_i}} = x_{m_i}^*|_{F_{m_i}}$.

Property $(3)$ guarantees the existence of a number $m_k \in \mathbb{N}$ such that $x_{m_k}^*|_{F_{m_{k-1}}} = 0$. It follows from property $(2)$ that there is a finite set $F_{m_k} \subset L$, with $F_{m_{k-1}} \subset F_{m_k}$, such that $f_{m_k}(x^*) \in [1 - \varepsilon_{kk},1+\varepsilon_{kk}]$ whenever $x^*|_{F_{m_k}} = x_{m_k}^*|_{F_{m_k}}$.

For each $k\in \mathbb{N}$ define $y^*_{m_k} \colon L \longrightarrow [-1,1]$ such that $y_{m_k}^*|_{F_{m_k}}=x_{m_k}^*|_{F_{m_k}}$ and $y_{m_k}^*(x)=0$ whenever $x \in L \setminus F_{m_k}$. Notice that $f_{m_k}(y_{m_k}^*) \in [1-\varepsilon_{kk},1+\varepsilon_{kk}]$ for every $k\in \mathbb{N}$. On the other hand, if $m_k<m_{k'}$ and $y_{m_k}^*(x) \neq 0$ then $x \in F_{m_k}$ (by the definition of $y_{m_k}^*$) and therefore $x^*_{m_{k'}}(x)=0$, so $y^*_{m_{k'}}(x)=0$. It follows that $y^*_{m_k}$ and $y^*_{m_{k'}}$ have disjoint supports. In particular, $$ \sup_{x \in L} \sum_{k=1}^l \abs{y_{m_k}^{\ast}(x)} \leq 1. $$ 

Let $\nu_1:= m_1 = 1$. Combining property $(1)$ with the fact that the functions $f_n$ are continuous in $[-1,1]^L$ and the functions $y_{m_n}^*$ converge to 0 in the product topology, we have that there exists $\nu_2 \in \mathbb{N}$ such that $$|f_{m_n}(y_{m_{\nu_1}}^*)| \leq \varepsilon_{12} \text{ and } |f_{m_{\nu_1}}(y_{m_n}^*)| \leq \varepsilon_{21} = \varepsilon_{12} \text{ for every }n \geq \nu_2.$$

Again, using the above, there exists a natural number $\nu_3 \geq \nu_2$ such that $$|f_{m_n}(y_{m_{\nu_1}}^*)| \leq \varepsilon_{13},\text{ }|f_{m_{\nu_1}}(y_{m_n}^*)| \leq \varepsilon_{31} = \varepsilon_{13}$$ and $$|f_{m_n}(y_{m_{\nu_2}}^*)| \leq \varepsilon_{23},\text{ }|f_{m_{\nu_2}}(y_{m_n}^*)| \leq \varepsilon_{32} = \varepsilon_{23}$$ for every $n \geq \nu_3.$

Suppose that we have $\nu_1 \leq \nu_2 \leq \cdots \leq \nu_p \in \mathbb{N}$ such that $$|f_{m_n}(y_{m_{\nu_j}}^*)| \leq \varepsilon_{jp} \text{ and }|f_{m_{\nu_j}}(y_{m_n}^*)| \leq \varepsilon_{pj} = \varepsilon_{jp} \text{ for every } j<p \text{ and every }n \geq \nu_p.$$

Then, there exists a natural number $\nu_{p+1} \geq \nu_p$ such that $$|f_{m_n}(y_{m_{\nu_j}}^*)| \leq \varepsilon_{j(p+1)} \text{ and }|f_{m_{\nu_j}}(y_{m_n}^*)| \leq \varepsilon_{(p+1)j} = \varepsilon_{j(p+1)} \text{ for every } j<p+1$$ and every $n \geq \nu_{p+1}.$

Since $f_{m_{\nu_i}}(y_{m_{\nu_i}}^*) \in [1-\varepsilon_{\nu_i \nu_i}, 1+\varepsilon_{\nu_i \nu_i}]$ for every $i$, we can write $f_{m_{\nu_i}}(y_{m_{\nu_i}}^*) = 1 + \eta_{\nu_i \nu_i}$ with $|\eta_{\nu_i \nu_i}| \leq \varepsilon_{\nu_i \nu_i}$.

On the other hand, if $k \neq i$, we have that $f_{m_{\nu_k}}(y_{m_{\nu_i}}^*) \in [-\varepsilon_{i k}, \varepsilon_{i k}]$, and we will write  $f_{m_{\nu_k}}(y_{m_{\nu_i}}^*) = \eta_{\nu_i \nu_k}$ with $|\eta_{\nu_i \nu_k}| \leq \varepsilon_{i k}$.

We take the subsequence $f_{n_k} := f_{m_{\nu_k}}$ for every $k \in \mathbb{N}$.

Now, let $\lambda_1, \ldots, \lambda_l \in \mathbb{R}$. We have that

\begin{eqnarray*} 
 \norm{ \sum_{k=1}^{l}\lambda_k f_{n_k} } & = &  \sup \set{\sum_{i = 1}^q \abs{\sum_{k=1}^{l}\lambda_k f_{n_k}(z_{i}^{*})} : q \in \mathbb{N}, \, z_i^{*} \in [-1,1]^L,\text{ }\sup_{x \in L} \sum_{i=1}^q \abs{z_i^{*}(x)} \leq 1 } \\ & \geq & \sum_{i = 1}^l \abs{\sum_{k=1}^{l}\lambda_k f_{m_{\nu_k}}(y_{m_{\nu_i}}^{*})}
= \sum_{i = 1}^l \abs{\lambda_i f_{m_{\nu_i}}(y_{m_{\nu_i}}^{*}) + \sum_{k \neq i}\lambda_k f_{m_{\nu_k}}(y_{m_{\nu_i}}^{*})} \\ & = & \sum_{i=1}^l \abs{ \lambda_i (1 +  \eta_{\nu_i \nu_i}) + \sum_{k \neq i}\lambda_k \eta_{\nu_i \nu_k} } = \sum_{i=1}^l \abs{\lambda_i + \sum_{k=1}^l \lambda_k \eta_{\nu_i \nu_k}} \\ & \geq & \sum_{i=1}^l \abs{\lambda_i} - \sum_{i=1}^l \sum_{k=1}^l \abs{\lambda_k} \abs{\eta_{\nu_i \nu_k}} = \sum_{i=1}^l \abs{\lambda_i} - \sum_{k=1}^l \abs{\lambda_k}\left(\sum_{i=1}^l \abs{\eta_{\nu_i \nu_k}}\right) \\ 
 &\geq & \sum_{i=1}^l \abs{\lambda_i} - \sum_{k=1}^l \abs{\lambda_k}\left(\varepsilon_{\nu_k \nu_k}+\sum_{i \neq k} \varepsilon_{ik}\right)
 \geq  (1 - \varepsilon) \sum_{k=1}^l \abs{\lambda_k}.
\end{eqnarray*}
\end{proof}

We are ready to prove Theorem~\ref{TheoremSubsequencesEquivalentTol1} from the introduction:

\begin{proof}[Proof of Theorem \ref{TheoremSubsequencesEquivalentTol1}]
Suppose that there is no subsequence equivalent to the canonical basis of $\ell_1$. Then, by Rosenthal's $\ell_1$-theorem \cite[Theorem 10.2.1]{AlbKal}, the sequence $(u_i)_{i \in \mathbb{N}}$ has a weakly Cauchy subsequence $(u_{i_k})_{k \in \mathbb{N}}$. Thus, the sequence $(y_n)_{n \in \mathbb{N}}$, with $y_n = u_{i_{2n+1}} - u_{i_{2n}}$ for every $n \in \mathbb{N}$, is weakly null and bounded.

Let us denote $T(x) = (T(x)_j)_{j \in \mathbb{N}} \in c_0$ for $x \in X$. Let $\tilde{T} \colon X \longrightarrow c_0$ be the map given by $\tilde{T}(x) = (T(x)_{i_{2k+1}})_{k \in \mathbb{N}}$.

Let $L = \mathcal{P}_{fin}(\omega) \setminus \set{\emptyset}$ be the set of the finite parts of $\omega$ without the empty set, and let $\Phi \colon FBL(L) \longrightarrow c_0$ be the map given by $$\Phi(f) = \big( f \left( \left( \chi_A( \set{1}) \right)_{A \in L} \right), f \left( \left( \chi_A( \set{2}) \right)_{A \in L} \right), \ldots \big) = \big( f \left( \left( \chi_A( \set{n}) \right)_{A \in L} \right) \big)_{n \in \mathbb{N}}$$ for every $f \colon [-1,1]^L \longrightarrow \mathbb{R} \in FBL(L)$, where for $A \in L$, $\chi_A \colon L \longrightarrow [-1,1]$ is the map given by $\chi_A(B) = 1$ if $B \subset A$ and $\chi_A(B) = 0$ if $B \not\subset A$.

By \cite[Lemma 2.2]{AMRA19}, $\Phi$ is a surjective Banach lattice homomorphism. Since $X$ is $\infty$-projective, there exists a bounded Banach lattice homomorphism $\ddot{T} \colon X \longrightarrow FBL(L)$ such that $\Phi \circ \ddot{T} = \tilde{T}$. We are going to find now $f_n$ and $x_n^\ast$ for the application of Lemma~\ref{lemmaauximproved}.

Let $f_n := \ddot{T}(y_n)$ for every $n \in \mathbb{N}$. The sequence $(f_n)_{n \in \mathbb{N}}$ converges pointwise to 0, since $(y_n)_{n \in \mathbb{N}}$ is weakly null. It follows from the equality $\Phi(f_n)=(\Phi \circ \ddot{T})(y_n) = \tilde{T}(y_n) = e_n$  and the definition of $\Phi$ that $$f_n \left( \left( \chi_A( \set{n}) \right)_{A \in L} \right) = \Phi(f_n)_n = e_n(n)= 1$$
for every $n \in \nat$. Set $x_n^*=\left(\chi_A ( \set{n}) \right)_{A\in L} \in [-1,1]^L$ for every $n\in \nat$.
Notice that if $F\subset L$ is finite, then  $x_n^* (S)=0$ whenever $n \notin \bigcup_{S\in F} S$, so condition (3) of Lemma~\ref{lemmaauximproved} is also satisfied.

We can now apply Lemma \ref{lemmaauximproved}, so for every $\varepsilon>0$ there is a subsequence $(f_{n_k})_{ k \in \nat }$ such that for every $l \in \mathbb{N}$ and for every $\lambda_1, \ldots, \lambda_l \in \mathbb{R}$, $$\norm{ \sum_{k=1}^{l}\lambda_k f_{n_k} } \geq (1 - \varepsilon)\sum_{k=1}^l |\lambda_k|.$$
On the other hand, since $\ddot{T}$ and $(y_n)_{n \in \mathbb{N}}$ are bounded, there are two constants $C,M>0$ such that 
$$ \norm{ \sum_{k=1}^{l}\lambda_k f_{n_k} } = \norm{ \ddot{T} \left( \sum_{k=1}^{l}\lambda_k y_{n_k} \right)}\leq C\norm{\sum_{k=1}^{l}\lambda_k y_{n_k}} \leq CM \sum_{k=1}^l |\lambda_k|.$$

Thus, $$(1 - \varepsilon)\sum_{k=1}^l |\lambda_k| \leq  \norm{ \sum_{k=1}^{l}\lambda_k f_{n_k} } \leq CM \sum_{k=1}^l |\lambda_k|,$$ so that $(f_{n_k})_{k \in \mathbb{N}}$ is equivalent to the canonical basis of $\ell_1$, and in consequence, $(y_{n_k})_{k \in \mathbb{N}}$ is also equivalent to the canonical basis of $\ell_1$, which is a contradiction.
\end{proof}

\begin{cor}
The Banach lattices $c_0$ and $l_p$ (for $2\leq p < \infty$) are not $\infty$-projective.
\end{cor}

\begin{proof}

On the one hand, the canonical basis $(u_i)_{i \in \mathbb{N}}$ of $c_0$ does not have subsequences equivalent to the canonical basis of $\ell_1$, and the identity map $T = id_{c_0}$ is a Banach lattice homomorphism such that $T(u_i) = e_i$ for every $i \in \mathbb{N}$, where $(e_i)_{i \in \mathbb{N}}$ is the canonical basis of $c_0$. On the other hand, the canonical basis $(u_i)_{i \in \mathbb{N}}$ of $l_p$ does not have subsequences equivalent to the canonical basis of $\ell_1$, and the formal inclusion $T$ of $l_p$ into $c_0$ is a Banach lattice homomorphism such that $T(u_i) = e_i$ for every $i \in \mathbb{N}$, where $(e_i)_{i \in \mathbb{N}}$ is the canonical basis of $c_0$.
\end{proof}

As we mentioned, the fact that $c_0$ is not $\infty$-projective was already proved in \cite[Theorem 2.4]{AMRA19}. The following is a corollary of Theorem~\ref{TheoremSubsequencesEquivalentTol1} in the context of free Banach lattices $FBL[E]$:

\begin{lem}\label{LemmaSubsequencesEquivalentTol1}
Let $E$ be a Banach space such that $FBL[E]$ is $\infty$-projective, and let $(u_i)_{i \in \mathbb{N}}$ be a bounded sequence of vectors in $E$. Suppose that there exists an operator $S \colon E \longrightarrow c_0$ such that $S(u_i) = e_i$ for every $i \in \mathbb{N}$, where $(e_i)_{i \in \mathbb{N}}$ is the canonical basis of $c_0$. Then there is a subsequence $(u_{i_k})_{k \in \mathbb{N}}$ equivalent to the canonical basis of $\ell_1$.
\end{lem}

\begin{proof}

Let $\phi \colon E \longrightarrow FBL[E]$ be the inclusion of $E$ into $FBL[E]$, and let $T \colon FBL[E] \longrightarrow c_0$ be the Banach lattice homomorphism given by the universal property of the free Banach lattice which extends the operator $S$.

The sequence $(\phi(u_i))_{i \in \mathbb{N}}$ is bounded in $FBL[E]$ and $T(\phi(u_i)) = S(u_i) = e_i$ for every $i \in \mathbb{N}$, so that applying Theorem \ref{TheoremSubsequencesEquivalentTol1} we have that $(\phi(u_i))_{i \in \mathbb{N}}$ has a subsequence $(\phi(u_{i_k}))_{k \in \mathbb{N}}$ equivalent to the canonical basis of $\ell_1$, which implies that $(u_{i_k})_{k \in \mathbb{N}}$ is a subsequence of $(u_i)_{i \in \mathbb{N}}$ equivalent to the canonical basis of $\ell_1$.
\end{proof}

We pass now to the proof of Theorem~\ref{TheoremSchurPropertyInBanachSpacesWithProjectiveFreeBanachLattice}, which states that $E$ has the Schur property when $FBL[E]$ is $\infty$-projective. Lemmas~\ref{LemmaEIsomorphictToASubspaceOfContinuousFunctionsWhenFBLIsProjective}, \ref{SCPCK} and \ref{LemmaToExtendOperatorsWhenFBLIsProjective} are necessary only to deal with the case when $E$ is nonseparable. The reader interested in the separable case may skip those lemmas and just apply Sobczyk's extension theorem \cite[Theorem 2.5.8]{AlbKal} where appropriate.

\begin{lem}\label{LemmaEIsomorphictToASubspaceOfContinuousFunctionsWhenFBLIsProjective}

Let $E$ be a Banach space. If $FBL[E]$ is $\infty$-projective, then $E$ is isomorphic to a subspace  of $C([-1,1]^{\Gamma})$ for some set $\Gamma$.

\end{lem}

\begin{proof}

Let $\Gamma$ be a dense subset of the unit ball $B_E$ of $E$. Let $B_{E^\ast}$ be the closed unit ball of the dual space $E^\ast$, endowed with the weak$^\ast$ topology. We have a surjective Banach lattice homomorphism $P \colon C([-1,1]^\Gamma) \To C(B_{E^\ast})$ given by $P(f)(x^\ast) = f((x^\ast(x))_{x\in \Gamma})$. This is just the composition operator with the continuous injection $x^\ast\mapsto (x^\ast(x))_{x\in\Gamma}$ from $B_{E^\ast}$ into $[-1,1]^\Gamma$. Let $\iota \colon E\To C(B_{E^\ast})$ be the canonical inclusion $\iota(x)(x^\ast) = x^\ast(x)$, and let $\hat{\iota} \colon FBL[E]\To C(B_{E^\ast})$ be the Banach lattice homomorphism given by the universal property of the free Banach lattice. Since $FBL[E]$ is supposed to be $\infty$-projective, there exists $\hat{T} \colon FBL[E]\To C([-1,1]^\Gamma)$ such that $P\circ\hat{T} = \hat{\iota}$. We take the restriction $T:= \hat{T}\vert_{E} \colon E \longrightarrow  C([-1,1]^{\Gamma})$. Notice that $PTx = \iota x$, and therefore $$\|Tx\| \geq \|PTx\| = \|\iota x\| = \| x\|$$ for every $x \in E$. This implies that $T$ gives an isomorphism of $E$ onto a subspace of $C([-1,1]^{\Gamma})$. 
\end{proof}

The following fact is well known in the context of a more general theory about Valdivia compacta, Plichko spaces and projectional skeletons (cf. for instance \cite{Kubis}), but we provide a short proof for the reader's convenience:

\begin{lem}\label{SCPCK}
	For every set $\Gamma$, the Banach space $C([-1,1]^\Gamma)$ has the separable complementation property. That is, for every separable subspace $G\subset C([-1,1]^\Gamma)$ there exists a separable complemented subspace $G_0$ of $C([-1,1]^\Gamma)$  such that $G\subset G_0$. 
\end{lem}

\begin{proof}
Let $S$ be a countable dense subset of $G$. By Mibu's theorem \cite[page 80, Theorem 4]{GenTopII}, for every $f\in S$ there exists a countable subset $\Gamma_f\subset \Gamma$ such that $f(x)=f(y)$ whenever $x|_{\Gamma_f} = y|_{\Gamma_f}$. The set $A= \bigcup_{f\in S}\Gamma_f$ is a countable set such that $f(x)=f(y)$ whenever $x|_A = y|_A$ and $f\in G$. The desired separable complemented subspace is $$G_0 = \left\{f\in C([-1,1]^\Gamma) : x|_A = y|_A \Rightarrow f(x)=f(y)\right\} \cong C([-1,1]^A).$$
The projection $P \colon C([-1,1]^\Gamma)\To G_0$ is given by $P(f)(x) = f(\tilde{x})$ where $\tilde{x}_i = x_i$ if $i\in A$ and $\tilde{x}_i =0$ if $i\not\in A$.
\end{proof}

\begin{lem}\label{LemmaToExtendOperatorsWhenFBLIsProjective}
Let $E$ be a Banach space such that $FBL[E]$ is $\infty$-projective, and let $F \subset E$ be a separable subspace. Every operator $S_0 \colon F \longrightarrow c_0$ can be extended to an operator $S \colon E \longrightarrow c_0$.
\end{lem}

\begin{proof}

By Lemma \ref{LemmaEIsomorphictToASubspaceOfContinuousFunctionsWhenFBLIsProjective}, there is an operator $T \colon E\To C([-1,1]^{\Gamma})$ that is an isomorphism onto its range, so that $G=T(F)$ is a separable subspace of $C([-1,1]^{\Gamma})$. By Lemma~\ref{SCPCK}, we can find a complemented separable subspace $G_0$ of $C([-1,1]^\Gamma)$ with $G\subset G_0$. Let $P \colon C([-1,1]^{\Gamma}) \longrightarrow G_0$ be the projection. If $S_0^{'} \colon G_0 \longrightarrow c_0$ is the extension of $S_0$ given by the Sobczyk's theorem, then $S := S_0^{'} \circ P \circ T \colon E \longrightarrow c_0$ is the desired operator.
\end{proof}

Theorem \ref{TheoremSchurPropertyInBanachSpacesWithProjectiveFreeBanachLattice} follows from the previous results:

\begin{proof}[Proof of Theorem \ref{TheoremSchurPropertyInBanachSpacesWithProjectiveFreeBanachLattice}]
If $E$ does not have the Schur property, then there is a weakly null sequence $(u_i)_{i \in \mathbb{N}}$ that does not converge to $0$ in norm. By passing to a subsequence we may assume that 0 is not in the norm closure of $\{u_i\}_{i \in \mathbb{N}}$. By the theorem of Kadets and Pe\l czy\'{n}ski \cite[Theorem 1.5.6]{AlbKal}, by passing to a further subsequence, we can suppose that $(u_i)_{i \in \mathbb{N}}$ is a basic sequence. We are going to see that there exists an operator $S \colon E \longrightarrow c_0$ such that $S(u_i) = e_i$ for every $i \in \mathbb{N}$, where $(e_i)_{i \in \mathbb{N}}$ is the canonical basis of $c_0$, and then by Lemma \ref{LemmaSubsequencesEquivalentTol1}, this will mean that $(u_i)_{i \in \mathbb{N}}$ has a subsequence equivalent to the canonical basis of $\ell_1$, a contradiction with the fact that it is weakly null.

Let $F = \overline{\text{span}}\set{u_i : i \in \mathbb{N}} \subset E$. For every $n \in \mathbb{N}$ let $u_n^* \colon F \longrightarrow \mathbb{R}$ be the $n$-th coordinate functional, given by $u_n^*(\sum_{i=1}^{\infty}\alpha_i u_i) = \alpha_n$, and let $S_0 \colon F \longrightarrow \ell_{\infty}$ be the map given by $S_0(x) = (u_n^*(x))_{n \in \mathbb{N}}$ for every $x \in F$. Since the sequence $(u_n^*)_{n \in \mathbb{N}}$ is weak$^*$-null, we have that $S_0(F) \subset c_0$.
On the other hand, $S_0(u_i) = e_i$ for every $i \in \mathbb{N}$. Now, since $F$ is separable and $FBL[E]$ is $\infty$-projective, applying Lemma \ref{LemmaToExtendOperatorsWhenFBLIsProjective} we can extend $S_0$ to an operator $S \colon E \longrightarrow c_0$ such that $S(u_i) = e_i$ for every $i \in \mathbb{N}$.
\end{proof}

 As a remark, along the first lines of the proof we justify that the Schur property is characterized by the property that every basic sequence contains a subsequence equivalent to the canonical basis of $\ell_1$. We may refer to \cite{HGM97} for a study of this kind of fact in a more general context.
 
 Finally, let us see that, in the category of nonseparable Banach spaces, the converse does not hold. By \cite[Theorem 1, e) and f)]{Hagler77}, there exist a separable Banach space $F$ and a bounded set $\Lambda$ in $F^*$ such that $E := \overline{\text{span}}(\Lambda)$ is nonseparable, has the Schur property and does not contain any copy of $\ell_1(\omega_1)$. Now, since for every set $\Gamma$ the space $[-1,1]^{\Gamma}$ is a continuous image of $\set{0,1}^m$ for some infinite cardinal number $m$, by \cite[Corollary 3]{Hagler75} we have that $E$ is not isomorphic to any subspace of $C([-1,1]^{\Gamma})$ for any set $\Gamma$, and then, by Lemma \ref{LemmaEIsomorphictToASubspaceOfContinuousFunctionsWhenFBLIsProjective}, we have that $FBL[E]$ cannot be $\infty$-projective.

\section{Projectivity of $C(K)$}\label{ProjectivityOfC(K)}

In \cite[Theorem 11.4]{dPW15} it is proved that for a compact subset $K$ of $\mathbb{R}^n$, $C(K)$, with the supremum norm, is a $1^+$-projective Banach lattice if, and only if, $K$ is an absolute neighbourhood retract of $\mathbb{R}^n$. In this section we prove a similar result for $K$ being a compact Hausdorff topological space not necessarily inside $\mathbb{R}^n$. We first recall some basic definitions and facts.

\begin{defn}\label{defanr}
We say that a compact space $K$ is an \textit{absolute neighbourhood retract} (ANR) if whenever $i \colon K\To X$ is a homeomorphism between $K$ and a subspace of the compact space $X$, there exist an open set $U$ and a continuous function $\phi \colon U\To K$ such that $i(K) \subset U \subset X$ and $\phi(i(k)) = k$ for all $k\in K$.
\end{defn}

\begin{lem}\label{LemmaNeighbourhoodRetract}
	In the situation of Definition~\ref{defanr}, there exist a continuous function $u \colon X \longrightarrow [0,1]$ and a continuous function $\varphi \colon X \setminus u^{-1}(0) \longrightarrow K$ such that $u(i(k)) = 1$ and $\varphi (i(k)) = k$ for every $k \in K$.
\end{lem}

\begin{proof}
 By the Urysohn's lemma, we can find a continuous function $u \colon X \longrightarrow [0,1]$ such that $u(i(k)) = 1$ for every $k \in K$, $u(x) = 0$ for every $x \in X \setminus U$, and $u(x) \in (0,1)$ for every $x \in U\setminus K$. Notice that $X \setminus u^{-1}(0)\subset U$, so the statement of the Lemma is satisfied.
\end{proof}

\begin{prop}\cite[Proposition 2.1]{ARA19}\label{quotientofprojective}
Let $P$ be a $1^+$-projective Banach lattice, $\mathcal{I}$ an ideal of $P$ and $T \colon P\To P/\mathcal{I}$ the quotient map. The quotient $P/\mathcal{I}$ is $1^+$-projective if, and only if, for every $\varepsilon>0$ there exists a Banach lattice homomorphism $S_\varepsilon \colon P/\mathcal{I}\To P$ such that $T\circ S_\varepsilon = id_{P/\mathcal{I}}$ and $\|S_\varepsilon\|\leq 1+\varepsilon$.
\end{prop}

\begin{lem}\label{LemmaFreeNormOfFunctionTimesDeltaLessInfiniteNormOfFunction}
Let $A$ be a set, $f\colon [-1,1]^A \longrightarrow \mathbb{R}$ a function, and $a \in A$. Then, the $FBL(A)$-norm of the pointwise product $f\cdot |\delta_a|$ is less than or equal to the supremum norm $\|f\|_\infty$. 
\end{lem}

\begin{proof}
\begin{eqnarray*} 
 \|f\cdot|\delta_a|\| & := &  \sup \set{\sum_{k = 1}^m \abs{f\cdot |\delta_a|(x_k^*)} :  m \in \mathbb{N}, \, x_k^* \in [-1,1]^A, \text{ }\sup_{x \in A} \sum_{k=1}^m \abs{x_k^{\ast}(x)} \leq 1 } \\ 
 &=& \sup \set{\sum_{k = 1}^m \abs{f(x_k^*)}\cdot\abs{\delta_a(x_k^*)} :  m \in \mathbb{N}, \, x_k^* \in [-1,1]^A, \text{ }\sup_{x \in A} \sum_{k=1}^m \abs{x_k^{\ast}(x)} \leq 1 } \\
 &\leq & \sup \set{\sum_{k = 1}^m \abs{f(x_k^*)}\cdot \abs{x_k^*(a)} :  m \in \mathbb{N}, \, x_k^* \in [-1,1]^A, \text{ } \sum_{k=1}^m \abs{x_k^{\ast}(a)} \leq 1 } \\
 & \leq & \sup \set{\max \set{|f(x_k^*)| : k = 1,\ldots, m} :  m \in \mathbb{N}, \, x_k^* \in [-1,1]^A, \text{ } \sum_{k=1}^m \abs{x_k^{\ast}(a)} \leq 1 } \\ 
 & \leq & \|f\|_{\infty}.
\end{eqnarray*}
\end{proof}

\begin{proof}[Proof of Theorem \ref{thmANR}]

In \cite[Proposition 11.7]{dPW15} it is proved that if $C(K)$ is $1^+$-projective, then $K$ is an ANR.

For the converse, let $X := [-1,1]^{B_{C(K)}}$, where $B_{C(K)} =  \{f\in C(K) : \|f\|_\infty\leq 1\}$ is the closed unit ball of the space of continuous functions. The map  $i \colon K \longrightarrow X$ given by $i(k) = (\gamma(k))_{\gamma \in B_{C(K)}}$ is an homeomorphism between $K$ and $i(K)$. By Lemma \ref{LemmaNeighbourhoodRetract} there exist a continuous function $u \colon X \longrightarrow [0,1]$ and a continuous function $\varphi\colon X \setminus u^{-1}(0) \longrightarrow K$ such that $u(i(k)) = 1$ and $\varphi (i(k)) = k$ for every $k \in K$.

By the universal property of the free Banach lattice, there is a Banach lattice homomorphism $T\colon FBL(B_{C(K)}) \longrightarrow C(K)$ that extends the inclusion $B_{C(K)}\hookrightarrow C(K)$. This is clearly a quotient map and its action is given by $Tf(k) = f(i(k))$ for every $f \in FBL(B_{C(K)})$, $k \in K$.

Since $FBL(B_{C(K)})$ is $1^+$-projective (\cite[Proposition 10.2]{dPW15}), by Proposition \ref{quotientofprojective}, it is enough to prove that there exists a Banach lattice homomorphism $S\colon C(K) \longrightarrow FBL(B_{C(K)})$ such that $T \circ S = id_{C(K)}$ and $\norm{S} \leq 1$.

Let $\bar{1} \in B_{C(K)}$ be the constant function equal to $1$, and let $v\colon \set{x \in X: x_{\bar{1}} \neq 0} \longrightarrow X$ be the map given by $v(x) = (v(x)_{\gamma})_{\gamma \in B_{C(K)}}$, where $$v(x)_{\gamma}= \left\{ \begin{array}{lcc}
              -1 & \text{if} &  \frac{x_{\gamma}}{x_{\bar{1}}} < -1, \\
              \frac{x_{\gamma}}{x_{\bar{1}}} &  \text{if}  &  \frac{x_{\gamma}}{x_{\bar{1}}}\in [-1,1]  ,\\
			  1 &\text{if} & \frac{x_{\gamma}}{x_{\bar{1}}}>1,
              \end{array}
    \right.$$ 

for every $x = (x_{\gamma})_{\gamma \in B_{C(K)}} \in X$ with $x_{\bar{1}} \neq 0$.

For a given $h \in C(K)$, define $f\colon X \longrightarrow \mathbb{R}$ by
$$f(x) := \begin{cases}
(h \circ \varphi \circ v)\cdot (u \circ v)(x) &\text{ if } x_{\bar{1}} \neq 0 \text{ and } u(v(x))\neq 0,\\
0 &\text{ otherwise }.
\end{cases}$$
Formally, we should call the function $f_h$ as it depends on $h$. But we omit the subindex for a more friendly notation (in fact the subindex would always be ``$h$'' along the proof). Notice also that $x_{\bar{1}}\neq 0$ is required for $x$ to be in the domain of $v$ and $u(v(x))\neq 0$ is required for $v(x)$ to be in the domain of $\varphi$.

The desired $S\colon C(K) \longrightarrow FBL(B_{C(K)})$ will be the map given by $Sh(x) = (f\cdot |\delta_{\bar{1}}|)(x)$ for every $h \in C(K)$, $x \in X$. The function $Sh$ is a real-valued function on $X=[-1,1]^{B_{C(K)}}$, and we will need to prove that, in fact, $Sh \in FBL(B_{C(K)})$. Once that is proved, the rest of properties required for $S$ are relatively easy to check: It is clear that $S$ is a linear map, and it preserves the lattice operations $\wedge$ and $\vee$. The fact that $\|S\|\leq 1$ comes from Lemma \ref{LemmaFreeNormOfFunctionTimesDeltaLessInfiniteNormOfFunction}: $$\norm{Sh} = \norm{f\cdot |\delta_{\bar{1}}|} \leq \norm{f}_{\infty} = \norm{(h \circ \varphi \circ v)(u \circ v))}_{\infty} \leq \norm{h}_{\infty}.$$
To see that $T \circ S = id_{C(K)}$, take $h \in C(K)$ and $k \in K$. Remember that $u(i(k))= 1$ and $\varphi(i(k))=k$ and notice that $i(k)_{\bar{1}} =1$ and $v(i(k)) = i(k)$ for every $k \in K$, so 
$$TSh(k) = Sh(i(k)) = (f\cdot |\delta_{\bar{1}}|)(i(k)) = h(\varphi(i(k)))\cdot u(i(k))  = h(k).$$

So we turn now to the remaining more delicate question whether $Sh \in FBL(B_{C(K)})$ for every $h \in C(K)$. Functions in the free Banach lattice must be continuous (in the product topology) and positively homogeneous (commute with multiplication by positive scalars). We check first that $Sh$ has these two properties. Clearly, $Sh$ is continuous on the open set $\set{x \in X : x_{\bar{1}}\neq 0,  u(v(x)) \neq 0}$ because $Sh$ is expressed there by the formula $(h\circ \varphi\circ v)\cdot (u\circ v)\cdot |\delta_{\bar{1}}|$. If $x_{\bar{1}}=0$, then for every $\varepsilon>0$ there is a neighbourhood $W$ such that $|f(y)|\cdot |y_{\bar{1}}|\leq \|h\|_\infty\cdot\varepsilon$ for all $y\in W$, so $Sh$ is also continuous at those $x$. If $x_{\bar{1}}\neq 0$ but $u(v(x))=0$, again, given $\varepsilon>0$, we can find a neighbourhood $W$ of $x$ where $y_{\bar{1}}\neq 0$ and $|f(y)|\cdot|y_{\bar{1}}|\leq \|h\|_\infty\cdot \varepsilon$ for all $y\in W$.  For positive homogeneity, on the one hand, if $x_{\bar{1}} \neq 0$, then  $v(\lambda x) = v(x)$ for every $0 < \lambda \leq 1$ and $x \in X$, while $|\delta_{\bar{1}}|$ is positively homogeneous. If $x_{\bar{1}} = 0$, then for every $0 < \lambda \leq 1$ we have that $Sh(\lambda x) = 0 = \lambda Sh(x)$.

Being continuous and positively homogeneous is a sufficient condition to belong to $FBL(A)$ in the case when $A$ is finite \cite{dPW15}. What we can deduce from this in the infinite case is that a function $g\colon [-1,1]^A\To \mathbb{R}$ belongs to $FBL(A)$ provided that is continuous, positively homogeneous and depends on finitely many coordinates \cite[Lemma 3.1]{AMRA19}. Depending on finitely many coordinates means that there is a finite subset $A_0 \subset A$ such that $g(x)=g(y)$ whenever $x|_{A_0} = y|_{A_0}$. We will prove that $Sh$ can be obtained as the limit, in the $FBL(B_{C(K)})$-norm, of a sequence of continuous and positively homogeneous functions that only depend on a finite number of coordinates from $[-1,1]
^{B_{C(K)}}$. This proves that $Sh\in FBL(B_{C(K)})$ since it is a closed space. 

Consider $L = \{x\in X : x_{\bar{1}} = 1\}\subset X$. Since the restriction $f|_L$ is a continuous function on the compact space $L$, by the Stone-Weierstrass' theorem, for every $n \in \mathbb{N}$ we can find a continuous function $f_n^+\in C(L)$ that depends only on finitely many coordinates of the cube $[-1,1]^{B_{C(K)}}$ such that  $$\norm{f|_L - f_n^+}_{\infty} < \frac{1}{n}.$$
Define $f_n\colon X \longrightarrow \mathbb{R}$ by
$$f_n(x) := \begin{cases}
f_n^+(v(x)) &\text{ if } x_{\bar{1}} \neq 0,\\
0 &\text{ otherwise }.
\end{cases}$$
It is clear that $f_n(\lambda x) = f_n(x)$ for all $0< \lambda \leq 1$ and $x \in X$, since $v(\lambda x) = v(x)$. Moreover, $f_n$ depends on finitely many coordinates because $f_n^+$ does so, and each coordinate of $v$ depends on two coordinates ($v(x)_\gamma$ only depends on $x_\gamma$ and $x_{\bar{1}}$). On the other hand, $f_n\cdot|\delta_{\bar{1}}|$ is continuous in $X$. This is because $f_n\cdot|\delta_{\bar{1}}|$ is continuous in $\set{x \in X : x_{\bar{1}} \neq 0}$ clearly, and, if $x_{\bar{1}} = 0$, then for every $\varepsilon > 0$ there is a neighbourhood $W$ such that $|f_n(y)|\cdot|y_{\bar{1}}| \leq \norm{f_n^+}_{\infty}\cdot\varepsilon$ for all $y \in W$. Thus, the functions $f_n\cdot|\delta_{\bar{1}}|$ are all continuous, positively homogeneous and depend on finitely many coordinates. It follows from the aforementioned  fact \cite[Lemma 3.1]{AMRA19} that $f_n\cdot|\delta_{\bar{1}}| \in FBL(B_{C(K)})$ for every $n \in \mathbb{N}$. It only remains to prove that $\norm{Sh - f_n\cdot |\delta_{\bar{1}}|} \rightarrow 0$ when $n \rightarrow \infty$. For this, first notice that $v(v(x)) = v(x)$ for all $x\in X$ with $x_{\bar{1}}\neq 0$. This is just because $v(x)_{\bar{1}} = 1$. From this, it follows that $f(x) = f(v(x))$ for all $x$ with $x_{\bar{1}}\neq 0$. This together with  Lemma \ref{LemmaFreeNormOfFunctionTimesDeltaLessInfiniteNormOfFunction} gives: 
\begin{eqnarray*}
\norm{Sh - f_n\cdot |\delta_{\bar{1}}|} & = & \norm{f\cdot |\delta_{\bar{1}}| - f_n\cdot |\delta_{\bar{1}}|} = \norm{(f-f_n)\cdot |\delta_{\bar{1}}|} \\ & \leq & \norm{f-f_n}_{\infty}\\ 
& = & \sup \set{|f(x)-f_n(x)| : x \in X}\\ 
& = & \sup \set{|f(x)-f_n(x)| : x \in X, x_{\bar{1}} \neq 0}\\ 
& = & \sup \set{|f(x)-f_n^+(v(x))| : x \in X, x_{\bar{1}} \neq 0}\\
& = & \sup \set{|f(v(x))-f_n^+(v(x))| : x \in X, x_{\bar{1}} \neq 0}\\
&\leq & \sup \set{|f(y)-f_n^+(y)| : y \in L} = \|f|_L - f_n^+\|_\infty <\frac{1}{n}.
\end{eqnarray*}
\end{proof}

\section{Problems}

Concerning the different variations of projectivity, it was already observed in \cite{dPW15} that if a Banach lattice $P$ has the property that every homomorphism into a quotient $T\colon P\To X/J$ can be lifted to a homomorphism $\hat{T}\colon  P\To X$, then $P$ is $\lambda$-projective for some $\lambda$. It is obvious that the class of $\infty$-projective Banach lattices is closed under renorming but the $1^+$-projective class is not. It was asked in \cite{dPW15} whether every $\infty$-projective Banach lattice is the renorming of a $1^+$-projective Banach lattice. But, in fact, we do not know a single example that separates these two classes, even by renorming.

\begin{prob}
	Find an equivalent norm on a $1^+$-projective Banach lattice that makes it $\infty$-projective but not $1^+$-projective. 
\end{prob}  

A natural candidate would be $FBL[E]$ with $E$ a suitable Banach space renorming of $\ell_1$.

Theorems~\ref{TheoremSubsequencesEquivalentTol1} and~\ref{TheoremSchurPropertyInBanachSpacesWithProjectiveFreeBanachLattice} suggest a large presence of the Banach space $\ell_1$ inside projective Banach lattices. This does not exclude other subspaces ($C[0,1]$ is $1^+$-projective and contains isometric copies of any separable Banach space) but we may at least ask:

\begin{prob}
	If $X$ is $\infty$-projective and infinite-dimensional, must $X$ contain a Banach subspace isomorphic to $\ell_1$?
\end{prob}

We proved that $E$ has the Schur property if $FBL[E]$ is $\infty$-projective. But the only positive case that we know is that $FBL[\ell_1(A)] = FBL(A)$ is $1^+$-projective.

\begin{prob}
	Is there a Banach space $E$ with the Schur property, not isometric to $\ell_1(A)$, for which $FBL[E]$ is $1^+$-projective? 	Is there a Banach space $E$ with the Schur property, not isomorphic to $\ell_1(A)$, for which $FBL[E]$ is $\infty$-projective?
\end{prob}

For a more complete picture, let us mention that other Banach lattices proved by B. de Pagter and A. W. Wickstead to be $1^+$-projective include all finite-dimensional Banach lattices (\cite[Theorem 11.1]{dPW15}), $\ell_1$ and any countable $\ell_1$-sum of separable $1^+$-projective Banach lattices (\cite[Theorem 11.11]{dPW15}).

\end{document}